\documentclass[12pt]{amsart}
\usepackage{amsfonts, amssymb, latexsym, hyperref,graphicx}
\usepackage[usenames, dvipsnames]{color}
\usepackage{ wasysym }
\usepackage{ytableau}
\usepackage{xcolor}
\usepackage{tikz}
\usetikzlibrary{calc, shapes, backgrounds,arrows,positioning,decorations.pathmorphing}

\newtheorem{theorem}{Theorem}[section]
\newtheorem{proposition}[theorem]{Proposition}
\newtheorem{lemma}[theorem]{Lemma}

\newtheorem{claim}[theorem]{Claim}
\newtheorem*{claim*}{Claim}

\newtheorem{Main Conjecture}[theorem]{Main Conjecture}
\newtheorem{conjecture}[theorem]{Conjecture}
\newtheorem{problem}[theorem]{Problem}
\theoremstyle{remark}

\newtheorem{example}[theorem]{Example}

\theoremstyle{plain}

\newcommand\FA{{\mathfrak{D}}} 
\newcommand{\FAring}{\mathfrak{P}}

\newcommand\rect{{{\sf Rect}}}

\newcommand{\cellsizeL}{19}
\newcommand{\cellsizeS}{14}
\newlength{\cellszL} 
\newsavebox{\cellL}
\sbox{\cellL}{\begin{picture}(\cellsizeL,\cellsizeL)
\put(0,0){\line(1,0){\cellsizeL}}
\put(0,0){\line(0,1){\cellsizeL}}
\put(\cellsizeL,0){\line(0,1){\cellsizeL}}
\put(0,\cellsizeL){\line(1,0){\cellsizeL}}
\end{picture}}
\newcommand\cellifyL[1]{\def\thearg{#1}\def\nothing{}%
\ifx\thearg\nothing
\vrule width0pt height\cellszL depth0pt\else
\hbox to 0pt{\usebox{\cellL} \hss}\fi%
\vbox to \cellszL{
\vss
\hbox to \cellszL{\hss$#1$\hss}
\vss}}
\newcommand\tableauL[1]{\vtop{\let\\\cr
\baselineskip -16000pt \lineskiplimit 16000pt \lineskip 0pt
\ialign{&\cellifyL{##}\cr#1\crcr}}}

\newlength{\cellszS} 
\newsavebox{\cellS}
\sbox{\cellS}{\begin{picture}(\cellsizeS,\cellsizeS)
\put(0,0){\line(1,0){\cellsizeS}}
\put(0,0){\line(0,1){\cellsizeS}}
\put(\cellsizeS,0){\line(0,1){\cellsizeS}}
\put(0,\cellsizeS){\line(1,0){\cellsizeS}}
\end{picture}}
\newcommand\cellifyS[1]{\def\thearg{#1}\def\nothing{}%
\ifx\thearg\nothing
\vrule width0pt height\cellszS depth0pt\else
\hbox to 0pt{\usebox{\cellS} \hss}\fi%
\vbox to \cellszS{
\vss
\hbox to \cellszS{\hss$#1$\hss}
\vss}}
\newcommand\tableauS[1]{\vtop{\let\\\cr
\baselineskip -16000pt \lineskiplimit 16000pt \lineskip 0pt
\ialign{&\cellifyS{##}\cr#1\crcr}}}

\definecolor{ltrGray}{RGB}{210, 210, 210}
\definecolor{ltGray}{RGB}{160, 160, 160}

\definecolor{ltrBlue}{RGB}{170, 196, 237}
\definecolor{ltBlue}{RGB}{78, 127, 203}

\newcommand\bad{hopeless}
\newcommand\good{optimistic}
\newcommand\Tab{{\sf SELT}}

\title{Shifted Edge labeled tableaux and Localizations}
\author{Colleen Robichaux}
\address{Dept.~of Mathematics, University of Illinois at Urbana-Champaign, Urbana, IL 61801}
\email{cer2@illinois.edu}
\begin{document}
\pagestyle{plain}

\mbox{}

\date{\today}

\begin{abstract}
We prove cases of a conjectural rule of H.~Yadav, A.~Yong, and the author
for structure coefficients of the D.~Anderson-W.~Fulton ring. 
In particular, we give a combinatorial description for certain localization coefficients of this ring, which is related to the equivariant cohomology of isotropic Grassmannians. 
\end{abstract}

\maketitle

\section{Introduction}\label{sec:intro}
This paper concerns a bijective combinatorics question arising from the author's previous work \cite{RYYedge}. We begin with some Schubert calculus motivation for the problem.

The \emph{symplectic group} ${\sf G}={\sf Sp}_{2n}({\mathbb C})$ is the automorphism group of a non-degenerate skew-symmetric bilinear form $\langle \cdot,\cdot\rangle$ on ${\mathbb C}^{2n}$. 
Let $Z={\sf LG}(n,2n)$
be the \emph{Lagrangian Grassmannian} of $n$-dimensional isotropic subspaces of ${\mathbb C}^{2n}$, where $V\subseteq {\mathbb C}^{2n}$ is \emph{isotropic} if $\langle v_1,v_2\rangle=0$ for all $v_1,v_2\in V$. 
The opposite Borel subgroup ${\sf B}_{-}\leq {\sf G}$ are those lower triangular matrices in ${\sf G}$. 
\emph{Schubert cells} are the orbits of under the action of ${\sf B}_{-}$ on $Z$, of which there are finitely
many. 
The Schubert cells and their Zariski closures, the \emph{Schubert varieties} $Z_{\lambda}=\overline{Z_{\lambda}^{\circ}}$, are indexed by strict partitions inside the shifted staircase partition 
\[\rho_n=(n,n-1,n-2,\ldots,3,2,1).\]
A \emph{strict partition} is an integer partition $\lambda=(\lambda_1>\lambda_2>\ldots>\lambda_{\ell})\in\mathbb{Z}_{>0}^\ell$. 
Let \[\mathcal{SP}_n:=\{\lambda\subseteq \rho_n \ | \ \lambda \text{ a strict partition}\}.\]
We identify $\lambda$ with its shifted shape,
the Young diagram with the $i$-th northmost row indented $i-1$ units from the west. 

\begin{example} We identify $\lambda=(5,3,2)\in \mathcal{SP}_5$ with the diagram below. 
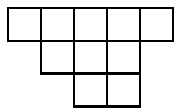
\[\begin{picture}(80,40)
\put(0,25){$\ytableausetup
{boxsize=1em}\begin{ytableau}
 \ & \  & \ & \  & \   \\
\none  & \  & \ & \ \\
\none  & \none & \ & \ \\
\end{ytableau}$}
\end{picture}\]
\end{example}

The maximal torus ${\sf T}$ are the diagonal matrices in ${\sf G}$. The \emph{equivariant Schubert
classes} $[Z_{\lambda}]_{\sf T}$ form a $H_{\sf T}^{*}(pt)$-module basis of $H^{*}_{\sf T}(Z)$. Define the structure coefficients by
\[ [Z_{\lambda}]_{\sf T}\cdot[Z_{\mu}]_{\sf T} = \sum_{\nu\subseteq \rho_n} L_{\lambda,\mu}^{\nu} [Z_{\nu}]_{\sf T},\]
where $L_{\lambda,\mu}^{\nu}\in H_{\sf T}^{*}(pt):=\mathbb{Z}[t_1,\ldots,t_n]$.
Using a theorem of Graham \cite{Graham},  
\begin{equation}\label{eq:lg}
    L_{\lambda,\mu}^{\nu}\in {\mathbb Z}_{\geq 0}[\alpha_1,\alpha_2,\ldots,\alpha_{n}], 
\end{equation}
where $\alpha_1=2{t}_1,\alpha_2={t}_2-{t}_1, \ldots,\alpha_{n}={t}_n-{t}_{n-1}$.

When $|\lambda|+|\mu|=|\nu|$,  $L_{\lambda,\mu}^{\nu}$ recover the ordinary structure coefficients $l_{\lambda,\mu}^{\nu}$ of $H^{*}(Z)$.
As determined by P.~Pragacz \cite{Pragacz},
these $l_{\lambda,\mu}^{\nu}$ are the structure coefficients for the multiplication of $Q$-Schur polynomials
of I.~Schur \cite{Schur}. Combinatorial rules for $l_{\lambda,\mu}^{\nu}$ are given by D.~Worley \cite{Worley}, B.~Sagan \cite{Sagan}, and J.~Stembridge \cite{Stembridge}. For a more in depth discussion of this story and its relation to the ordinary and maximal Grassmannian settings, see \cite{RYYedge}.

\begin{problem}\label{problem:OG}
Give a combinatorial rule for $L_{\lambda,\mu}^{\nu}$.
\end{problem}

For $\lambda\in {\mathcal {SP}}_n$, let $\sigma_{\lambda}={\sf Pf}(c_{\lambda_i,\lambda_j})$ be the Pfaffian where 
\[c_{p,q}=\displaystyle\sum_{0\leq a\leq b \leq q}(-1)^b\left(\binom{b}{a}+\binom{b-1}{a}\right)z^a c_{p+b-a} c_{q-b}.\] 
If $\ell = \ell(\lambda)$ is odd, take $\lambda_{\ell+1}=0$.
In \cite{AF} D.~Anderson-W.~Fulton study the ${\mathbb Z}[z]$-algebra 
\[{\FAring}=\mathbb{Z}[z,c_1,c_2,\ldots]/(c_{p,p}=0, \forall p>0)\]
 with basis $\{\sigma_\lambda\}_{\lambda\in{\mathcal {SP}}_n}$ over ${\mathbb Z}[z]$. Define structure coefficients $\FA_{\lambda,\mu}^{\nu}$ by
\[\sigma_{\lambda}\cdot \sigma_{\mu} =\sum_{\nu\in\mathcal {SP}_n} \FA_{\lambda,\mu}^{\nu}\sigma_{\nu}.\]
D.~Anderson-W.~Fulton make the following connection to $H^{*}_{\sf T}(Z)$:
\begin{equation}
\label{eqn:FultonAndersonconnect}
L_{\lambda,\mu}^{\nu}(\alpha_1\mapsto z, \alpha_2\mapsto 0,\ldots,\alpha_n\mapsto 0)={\mathfrak D}_{\lambda,\mu}^{\nu}.
\end{equation}
Taking 
$\Delta(\nu;\lambda,\mu):=|\lambda|+|\mu|-|\nu|$ and $L(\nu;\lambda,\mu):=\ell(\lambda)+\ell(\mu)-\ell(\nu)$, let
\[{\mathfrak d}_{\lambda,\mu}^{\nu}:=\frac{\FA_{\lambda,\mu}^{\lambda}}{2^{L(\lambda;\lambda,\mu)-\Delta(\lambda;\lambda,\mu)}z^{\Delta(\lambda;\lambda,\mu)}}.\]
In fact, ${\mathfrak d}_{\lambda,\mu}^{\nu}\in \mathbb{Z}_{\geq0}$.
With H.~Yadav and A.~Yong in \cite{RYYedge}, the author proposes a rule for ${\mathfrak d}_{\lambda,\mu}^{\nu}$ in terms of \emph{shifted edge labeled tableaux}, with $d_{\lambda,\mu}^{\nu}$ denoting the number of tableaux satisfying this rule. By Equation (\ref{eqn:FultonAndersonconnect}), this is a conjectural rule for a specialization of Problem \ref{problem:OG}.

\begin{conjecture}\cite[Conjecture 10.1]{RYYedge}
\label{conj:C=D} ${\mathfrak d}_{\lambda,\mu}^{\nu}=d_{\lambda,\mu}^{\nu}$.
\end{conjecture}

We contribute to the partial results of \cite{RYYedge} for Conjecture 10.1. 
For $0\leq m\leq n\in\mathbb{Z}_{>0}$, let \[\rho_{n,m}=(n,n-1,\ldots,n-m+1).\] The following generalizes \cite[Theorem 10.6]{RYYedge}, for which only a proof sketch was included.

\begin{theorem}\label{thm:locCoeffSkewSC}
Suppose $\ell(\mu)=m\leq\ell(\lambda)=n$ such that for $\rho_{n,m}\subseteq \mu$. Then 
\[{\mathfrak d}_{\lambda,\mu}^{\lambda}=d_{\lambda,\mu}^{\lambda},\]
where $d_{\lambda,\mu}^{\lambda}=2^{\binom{n}{2}-\binom{n-m}{2}}$ when $\mu=\rho_{n,m}$ and $0$ otherwise.
\end{theorem}

  In the case $\mu=\rho_{n,m}$, Theorem \ref{thm:locCoeffSkewSC} highlights an intriguingly simple enumeration. We focus on this case in particular as an enumerative combinatorics problem.  

\begin{example}\label{ex:shadeBij} When $\mu=\rho_{n,m}$, Theorem \ref{thm:locCoeffSkewSC} shows that ${\mathfrak d}_{\lambda,\mu}^{\lambda}$ is computed by subsets of boxes in the first $n-1$ columns of $\mu$. Further, Theorem \ref{thm:locCoeffSkewSC} states that these subsets are in bijection with certain shifted edge labeled tableaux.
Below is the shifted edge labeled tableau determined by the given shaded blue subset 
for $\lambda=(5,4,3,1)$ and $\mu=\rho_{4,2}$.
 
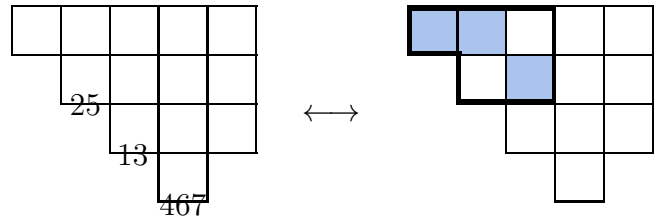
\[\begin{picture}(250,80)
\put(0,60){$\ytableausetup
{boxsize=1.5em}\begin{ytableau}
 \ & \  & \  & \ & \   \\
\none  & \ & \ & \ & \ \\
\none  & \none & \ & \ & \ \\
 \none & \none & \none & \ 
\end{ytableau}$}
\put(22,37){$2 5$}
\put(40,18){$1 3$}
\put(56,-1){$4 6 7$}
\put(110,35){$\longleftrightarrow$}
\put(150,60){$\ytableausetup
{boxsize=1.5em}\begin{ytableau}
 *(ltrBlue)\ & *(ltrBlue)\  & \  & \ & \  \\
\none  & \ & *(ltrBlue)\ & \ & \ \\
\none  & \none & \ & \ & \ \\
 \none & \none & \none & \ 
\end{ytableau}$}
\linethickness{0.5mm}
\put(150,77.5){{\line(1,0){55.8}}}
 \put(150,60.5){{\line(1,0){20}}}
\put(168.8,42){{\line(1,0){37}}}
\put(150.6,59.6){{\line(0,1){18.8}}}
\put(169.2,41.1){{\line(0,1){18.9}}}
\put(205.2,41.1){{\line(0,1){36}}}
\end{picture}\]  
\end{example}
\section{Combinatorial Background}\label{sec:shiftededge}
We first recall the {shifted edge labeled tableaux} defined in \cite{RYYedge}.

\subsection{Shifted edge labeled tableaux} For $\lambda\subseteq \nu$, the skew shape $\nu/\lambda$ is those boxes
of $\nu$ not in $\lambda$. A \emph{diagonal edge} of $\nu/\lambda$ 
is the southern edge of a \emph{diagonal box} of $\nu$, a box in matrix position $(i,i)$. 
When $\lambda=\emptyset$ we say $\nu=\nu/\lambda$ is a \emph{straight shape}.

\begin{example}\label{ex:skewShape}For $\nu=(5,3,2)$ and $\lambda=(3,2)$, the skew shape $\nu/\lambda$ consists of the five unshaded boxes shown below. The three diagonal edges of $\nu/\lambda$ are bolded in blue.
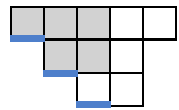
\[\begin{picture}(80,35)
\put(0,20){$\ytableausetup
{boxsize=1em}\begin{ytableau}
 *(ltrGray)\ & *(ltrGray)\  & *(ltrGray)\ & \  & \   \\
\none  & *(ltrGray)\  & *(ltrGray)\ & \ \\
\none  & \none & \ & \ \\
\end{ytableau}$}
\linethickness{0.8mm}
\put(0,20){\textcolor{ltBlue}{\line(1,0){13}}}
\put(12.5,7){\textcolor{ltBlue}{\line(1,0){13}}}
\put(25,-5){\textcolor{ltBlue}{\line(1,0){13}}}
\end{picture}\]
\end{example}

A \emph{shifted edge labeled tableau} of shape $\nu/\lambda$ is a filling of the boxes and diagonal edges of $\nu/\lambda$ with labels $[n]:=\{1,2,3,\ldots,n\}$ such that:
\begin{itemize}
\item[(i)] Each box of $\nu/\lambda$ contains exactly one label.
\item[(ii)] A diagonal edge of $\nu/\lambda$ contains a (possibly empty) set of labels.
\item[(iii)] Each $i\in[n]$ appears exactly once.
\item[(iv)] Labels strictly increase west to east across rows and down columns. Each label on a diagonal edge is strictly larger
than any label directly north of it.
\end{itemize}

\begin{example}For $\nu$ and $\lambda$ as in Example \ref{ex:skewShape}, only the leftmost tableau below is a {shifted edge labeled tableau} of shape $\nu/\lambda$. Reading left to right, the remaining tableaux violate (ii), (iii), and (iv), respectively.
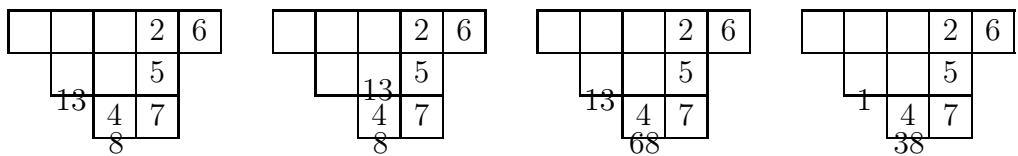
\[\begin{picture}(370,50)
\put(0,35){$\ytableausetup
{boxsize=1.3em}\begin{ytableau}
 \ & \  & \ & 2  & 6   \\
\none  & \  & \ & 5 \\
\none  & \none & 4 & 7 \\
\end{ytableau}$}
\put(18,13){$1 3$}
\put(38,-3){$8$}
\put(100,35){$\ytableausetup
{boxsize=1.3em}\begin{ytableau}
 \ & \  & \ & 2  & 6   \\
\none  & \  & \ & 5 \\
\none  & \none & 4 & 7 \\
\end{ytableau}$}
\put(134,17){$1 3$}
\put(138,-3){$8$}
\put(200,35){$\ytableausetup
{boxsize=1.3em}\begin{ytableau}
 \ & \  & \ & 2  & 6   \\
\none  & \  & \ & 5 \\
\none  & \none & 4 & 7 \\
\end{ytableau}$}
\put(218,13){$1 3$}
\put(235,-3){$6 8$}
\put(300,35){$\ytableausetup
{boxsize=1.3em}\begin{ytableau}
 \ & \  & \ & 2  & 6   \\
\none  & \  & \ & 5 \\
\none  & \none & 4 & 7 \\
\end{ytableau}$}
\put(321,13){$1$}
\put(335,-3){$3 8$}
\end{picture}\]	
\end{example}
Let ${\sf SELT}(\nu/\lambda,n)$ be the set of
all such tableaux. Restricting to tableaux satisfying only (i), (iii) and (iv) results in \emph{shifted standard Young tableaux}, as described in \cite{Worley}. 

An \emph{inner corner} ${\sf c}$ of $\nu/\lambda$ is a maximally southeast box of $\lambda$. For $T \in {\sf SELT}(\nu/\lambda,n)$, we define a \emph{jeu de taquin slide} ${\sf jdt}_{\sf c}(T)$, by the following. First place $\bullet$ in ${\sf c}$, and apply one of the following {slides} determined by the labels around ${\sf c}$:
\begin{itemize}
\item[(1)] $\tableauS{\bullet & a\\ b}\mapsto \tableauS{b & a\\ \bullet }$ \ (if $b<a$, or $a$ does not exist)
\smallskip
\item[(2)] $\tableauS{\bullet & a\\ b}\mapsto \tableauS{a & \bullet\\ b}$ \ (if $a<b$, or $b$ does not exist)
\item[(3)] $\begin{picture}(30,20)\put(0,0){$\tableauS{\bullet & a}$}
\put(3,-5){$S$}
\end{picture} \mapsto
\begin{picture}(30,20)\put(0,0){$\tableauS{a & \bullet }$}
\put(3,-5){$S$}
\end{picture}$ (if $a<\min(S)$ or $S=\emptyset$)
\item[(4)] $\begin{picture}(30,20)\put(0,0){$\tableauS{\bullet & a}$}
\put(3,-5){$S$}
\end{picture} \mapsto
\begin{picture}(30,20)\put(0,0){$\tableauS{s & a }$}
\put(3,-5){$S'$}
\end{picture}$ (if $s:=\min(S)<a$, or $a$ does not exist, where $S':=S\setminus \{s\}$).
\end{itemize}
Repeat this process on each new box occupied by $\bullet$ until $\bullet$ arrives at a diagonal edge
of $\lambda$ (\emph{i.e.} (4) is used) or a box that has no labels directly south or east of it. Then obtain ${\sf jdt}_{\sf c}(T)$ by removing $\bullet$ from the resulting tableau.

\begin{example}\label{ex:jdtpath} Below is the computation of ${\sf jdt}_{(1,2)}(T)$.
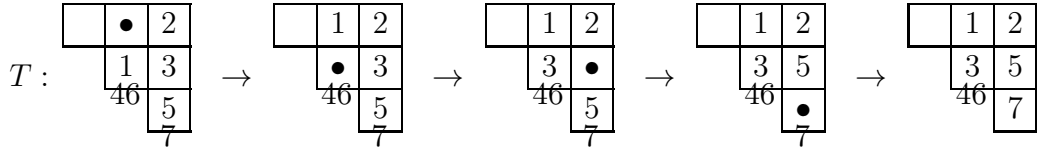
	\[\begin{picture}(380,55)
\put(-10,25){$T:$}
\put(10,40){$\ytableausetup
{boxsize=1.3em}\begin{ytableau}
 \ & \bullet  & 2   \\
\none    & 1 & 3 \\
\none  & \none  & 5 \\
\end{ytableau}$}
\put(28,18){$4 6$}
\put(47,2){$ 7$}
\put(70,25){$\rightarrow$}
\put(90,40){$\ytableausetup
{boxsize=1.3em}\begin{ytableau}
 \ & 1  & 2   \\
\none    & \bullet & 3 \\
\none  & \none  & 5 \\
\end{ytableau}$}
\put(108,18){$4 6$}
\put(127,2){$ 7$}
\put(150,25){$\rightarrow$}
\put(170,40){$\ytableausetup
{boxsize=1.3em}\begin{ytableau}
 \ & 1  & 2   \\
\none    & 3 & \bullet \\
\none  & \none  & 5 \\
\end{ytableau}$}
\put(188,18){$4 6$}
\put(207,2){$ 7$}
\put(230,25){$\rightarrow$}
\put(250,40){$\ytableausetup
{boxsize=1.3em}\begin{ytableau}
 \ & 1  & 2   \\
\none    & 3 & 5 \\
\none  & \none  & \bullet \\
\end{ytableau}$}
\put(268,18){$4 6$}
\put(287,2){$ 7$}
\put(310,25){$\rightarrow$}
\put(330,40){$\ytableausetup
{boxsize=1.3em}\begin{ytableau}
 \ & 1  & 2   \\
\none    & 3 & 5 \\
\none  & \none  & 7 \\
\end{ytableau}$}
\put(348,18){$4 6$}
\end{picture}\]	
\end{example}

The \emph{row rectification} ${{\sf Rect}}(T)$ of $T\in {\sf SELT}(\nu/\lambda,n)$ is defined iteratively: Choose the southmost inner corner ${\sf c}_0$ of $\nu/\lambda$ and compute
$T_1:={\sf jdt}_{{\sf c}_0}(T)$ with shape $\nu^{(1)}/\lambda^{(1)}$. Now let ${\sf c}_1$ be a southmost inner corner of $\nu^{(1)}/\lambda^{(1)}$ and compute
$T_2:={\sf jdt}_{{\sf c}_1}(T_1)$. Repeat $|\lambda|$ times, arriving at a straight shape tableau ${{\sf Rect}}(T)$.

Let $S_{\mu}$ be the superstandard tableau of $\mu$, obtained by filling the boxes of $\mu$ in English reading order
with $[n]$, where $n=|\mu|$. 
Define
\[d_{\lambda,\mu}^{\nu}:=\#\{T\in {\sf SELT}(\nu/\lambda,|\mu|): {\rect}(T)=S_{\mu}\}.\]

\begin{example}
Suppose $\lambda=(2,1), \mu=(3,2), \nu=(3,2)$. Below are the only shifted edge labeled tableaux that rectify to $S_{\mu}$, so $d_{\lambda,\mu}^{\nu} = 2$.

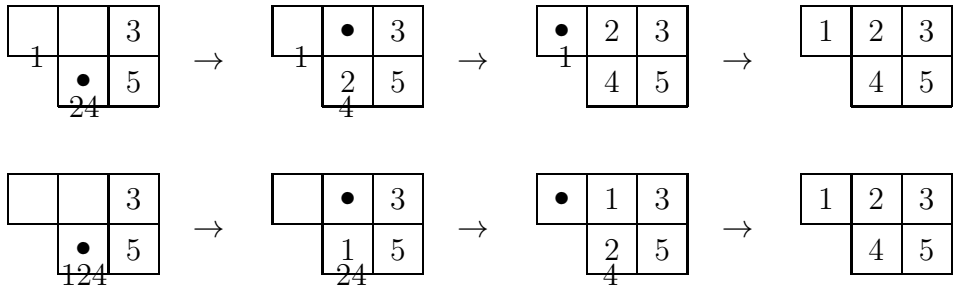
\[\begin{picture}(270,50)
\put(-50,35){$\tableauL{{\ }&{\ }&{3}\\ & {\bullet} & 5} $}
\put(-42,31){$1$}
\put(-27,12){$24$}
\put(20,30){$\rightarrow$}
\put(50,35){$\tableauL{{\ }&{\bullet}&{3}\\ & {2 } & 5}$}
\put(58,31){$1$}
\put(75,12){$4$}
\put(120,30){$\rightarrow$}
\put(150,35){$\tableauL{{\bullet }&{2}&{3}\\ & {4} & 5}$}
\put(158,31){$1$}
\put(220,30){$\rightarrow$}
\put(250,35){$\tableauL{{1}&{2}&{3}\\ & {4} & 5}$}
\end{picture}\]
\[\begin{picture}(270,40)
\put(-50,25){$\tableauL{{\ }&{\ }&{3}\\ & {\bullet} & 5}$}
\put(-30,2){$124$}
\put(20,20){$\rightarrow$}
\put(50,25){$\tableauL{{\ }&{\bullet}&{3}\\ & {1} & 5}$}
\put(74,2){$24$}
\put(120,20){$\rightarrow$}
\put(150,25){$\tableauL{{\bullet }&{1}&{3}\\ & {2} & 5}$}
\put(175,2){$4$}
\put(220,20){$\rightarrow$}
\put(250,25){$\tableauL{{1}&{2}&{3}\\ & {4} & 5}$}
\end{picture}\]
\end{example}

\subsection{Excited Young diagrams}
For $\lambda,\mu\in {\mathcal SP}_n$, place $+$'s in the shape of $\lambda$ inside $\mu$. Call this the \emph{initial diagram} of $\lambda$ in $\mu$. 
Define the following local move on the $+$'s in $\mu$:
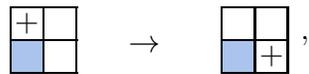
\[\begin{picture}(100,30)
\put(0,15){$\ytableausetup
{boxsize=1em}\begin{ytableau}
 + &  \   \\
 *(ltrBlue)\ &  \   
\end{ytableau}$}
\put(45,10){$\rightarrow$}
\put(80,15){$\ytableausetup
{boxsize=1em}\begin{ytableau}
 \ &  \   \\
 *(ltrBlue)\ &  +   
\end{ytableau}$}
\put(110,15){$,$}
\end{picture}\]
where the shaded box either does not exist in $\mu$ or is a box in $\mu$ unoccupied by a $+$.
An \emph{excited Young diagram} (EYD) of $\lambda$ in $\mu$ is a configuration of $+$'s formed by successive applications of the above local move on the initial diagram $\lambda$ in $\mu$.
Let $\mathcal{E}_{\mu}(\lambda)$ denote the set of all EYDs of $\lambda$ in $\mu$. If $\lambda\not\subseteq\mu$, then $\mathcal{E}_{\mu}(\lambda)=\emptyset$.

\begin{example}\label{ex:Eyd1}
For $\lambda=(2,1), \mu=(5,3,2)$, below are the EYDs in $\mathcal{E}_{\mu}(\lambda)$, where the leftmost is the initial diagram:
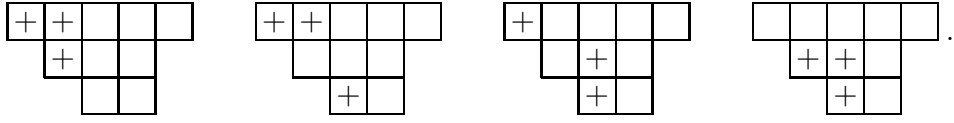
\begin{gather*}
\tableauS{{+}&{+}&{\ }&{\ }&{\ }\\  & {+}&{\ }&{\ }\\  & & {\ } & {\ }}
\qquad
\tableauS{{+}&{+}&{\ }&{\ }&{\ }\\  & {\ }&{\ }&{\ }\\  & & {+}& {\ }}
\qquad
\tableauS{{+}&{\ }&{\ }&{\ }&{\ }\\  & {\ }&{+ }&{\ }\\  & & {+}& {\ }}
\qquad
\tableauS{{\ }&{\ }&{\ }&{\ }&{\ }\\  & {+ }&{+ }&{\ }\\  & & {+}& {\ }}\ .
\end{gather*}
\end{example}

The following is derived from results of T.~Ikeda-H.~Naruse \cite{Ikeda.Naruse} using Equation (\ref{eq:lg}).
\begin{lemma}\cite[Lemma 10.5]{RYYedge}\label{lem:locAF}
${\mathfrak{d}}_{\lambda,\mu}^{\lambda} = \#\mathcal{E}_{\rho_{\ell(\lambda)}}(\mu)\times 2^{|\mu|-\ell(\mu)}$.
\end{lemma}

Using this result, 
\cite{RYYedge} obtains the following partial result towards Conjecture \ref{conj:C=D}.
\begin{theorem}\cite[Theorem 10.3]{RYYedge}\label{thm:locPieri}
$d_{\lambda,(p)}^{\lambda}=\binom{\ell(\lambda)}{p}2^{p-1}={\mathfrak d}_{\lambda,(p)}^{\lambda}$.
\end{theorem}

\section{Proof of Theorem \ref{thm:locCoeffSkewSC}}
For $T\in \Tab(\nu/ \lambda,|\mu|)$, let $T(i,j)$ be the entry in box $(i,j)$ (in matrix coordinates) and $E_i(T)$ be the set of entries on the $i$th diagonal edge of $\nu$. Additionally, let 
\[{\sf col}_k(T)=  E_k(T)\cup \{ T(i,k) \ | (i,k)\in \nu/\lambda \text{ where } i\in[k]\}.\]
Define $U_{n,m}\in\Tab(\rho_n / \rho_n,|\rho_{n,m}|)$ by the property
$E_i(U_{n,m})={\sf col}_i(S_{\rho_{n,m}})$ for each $i\in[n]$. That is, the labels on the $i$th diagonal edge of $U_{n,m}$ are precisely the labels appearing in the $i$th column of $S_{\rho_{n,m}}$.
For $T\in \Tab(\rho_n / \rho_n,|\rho_{n,m}|)$ and $I\subseteq E_{h}(T)$ where $h\in[n-1]$, define the \emph{$I$-slide of $T$}, denoted ${\sf Sl}_{I}(T)$, by 
  \[E_i({\sf Sl}_{I}(T)):=\begin{cases}
	E_i(T) &\text{if } i\in [n]\setminus\{h,h+1\},\\
    E_{h}(T)\setminus {I} &\text{if } i=h,\\
	E_{h+1}(T)\cup {I} &\mbox{if } i=h+1. 
	\end{cases}\]
By definition, ${\sf Sl}_{I}(T)\in \Tab(\rho_n / \rho_n,|\rho_{n,m}|)$.

\begin{example}\label{ex:slDef}
Let $n=4$ and $m=3$. Taking $I=\{6\}\subseteq E_3(U_{4,3})=\{3,6,8\}$, below we illustrate ${\sf Sl}_I(U_{4,3})$.
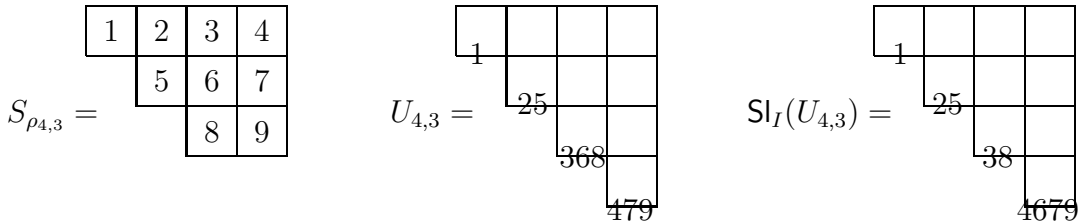
\[\begin{picture}(420,90)
\put(0,40){$S_{\rho_{4,3}}=$
}
\put(30,65){$\tableauL{ 1 & 2 & 3 & 4 \\ & 5  & 6 & 7 &\\ &  &  8 &9 }$
}
\put(145,40){$U_{4,3}=$}
\put(170,65){$\tableauL{ \ & \ & \ & \ \\ & \  & \ & \ \ &\\ &  &  \ &\ \\ & &  & \ }$}
\put(175,62){$1$}
\put(193,43){$2 5$}
\put(209,23){$3 6 8$}
\put(227,03){$4 7 9$}
\put(280,40){${\sf Sl}_I(U_{4,3})=$}
\put(328,65){$\tableauL{ \ & \ & \ & \ \\ & \  & \ & \ \ &\\ &  &  \ &\ \\ & &  & \ }$}
\put(335,62){$1$}
\put(350,43){$2 5$}
\put(369,23){$3 8$}
\put(382,03){$4 6 7 9$}
\end{picture}\]      
\end{example}

We say $T\in \Tab(\rho_n / \lambda,|\rho_{n,m}|)$ is \emph{\bad{}} if there exists some $j\in {\sf col}_\ell(S_{\rho_{n,m}})\cap {\sf col}_{\ell'}(T)$ where $\ell'<\ell$.

\begin{proposition}\label{prop:badEdgeOrder}
If $T\in \Tab(\rho_n / \lambda,|\rho_{n,m}|)$ is \bad{}, ${\rect}(T)\neq  S_{\rho_{n,m}}$.
\end{proposition}
\begin{proof}
This follows by the definition of row rectification.
\end{proof}

If $T\in \Tab(\rho_n / \rho_n,|\rho_{n,m}|)$ is not \bad{}, we say $T$ is \emph{\good{}}. Then \[T=({\sf Sl}_{I_{n-1}}\circ\ldots\circ{\sf Sl}_{I_1})(U_{n,m})\]
where $I_\ell\subseteq E_\ell(({\sf Sl}_{I_{\ell-1}}\circ\ldots\circ{\sf Sl}_{I_1})(U_{n,m}))$.
We call $\mathcal{I}(T):=(I_1,I_2,\ldots,I_{n-1})$ a \emph{slide decomposition} of $T$. For example taking $T={\sf Sl}_I(U_{4,3})$ from Example \ref{ex:slDef}, $\mathcal{I}(T)=(\emptyset,\emptyset,\{6\})$. Note that by the condition $I_\ell\subseteq E_\ell(({\sf Sl}_{I_{\ell-1}}\circ\ldots\circ{\sf Sl}_{I_1})(U_{n,m}))$, $\mathcal{I}(T)$ is unique.

Suppose $T\in\Tab(\rho_n / \rho_n,|\rho_{n,m}|)$ is \good{}. For some $h\in[n-1]$ where $i>h+1$ implies $\mathcal{I}(T)_i=\emptyset$, choose $j\in I\subseteq E_h(T)$. 
Define the operator ${\sf shift}_j$ on $\widetilde{T}\in\{{\sf Sl}_I(T)_i\}_{i\in[\binom{n+1}{2}]}$ such that if $j\in {\sf col}_{h+1}(\widetilde{T})$ and $\widetilde{T}(h,h)<j$:  
\[{\sf shift}_j( \widetilde{T}(r,c)):=\begin{cases}
	\widetilde{T}(r,c) &\text{if } c\neq h+1,\\
	\widetilde{T}(r,c) &\text{if } c=h+1 \mbox{ and } \widetilde{T}(r,c)<j,\\
    \widetilde{T}(r+1,c) &\text{if } c=h+1, \widetilde{T}(r,c)\geq j, \mbox{ and }  r\leq h,\\
    \min\{E_{h+1}(\widetilde{T})\} &\text{if } c=h+1, \widetilde{T}(r,c)\geq j, \mbox{ and }  r=h+1. 
	\end{cases}\]
\[{\sf shift}_j (E_{i}(\widetilde{T})):=\begin{cases}
	 E_{i}(\widetilde{T}) &\text{if } i\in[n]\setminus\{h,h+1\},\\
    E_{h}(\widetilde{T})\cup \{j\} &\text{if } i=h,\\
	E_{h+1}(\widetilde{T})\setminus \min_{a\geq j}\{a\in E_{h+1}(\widetilde{T}) \} &\mbox{if } i=h+1.
	\end{cases}\]
Otherwise, 
${\sf shift}_j$ acts trivially. In short, when nontrivial, ${\sf shift}_j$ moves $j$ from column $h+1$ to $E_{h}(T)$ and moves labels in column $h+1$ up accordingly. 
For $J=\{j_1<\ldots<j_{\ell}\}\subseteq I$ define \[{\sf shift}_J(\widetilde{T}):={\sf shift}_{j_{\ell}}\circ\ldots\circ {\sf shift}_{j_{1}}(\widetilde{T}).\]
By construction, ${\sf shift}_J(\widetilde{T})\in  \Tab(\rho_n / \lambda,|\rho_{n,m}|)$ for some $\lambda\subseteq \rho_n $. 

\begin{example}
Taking $T$ as below and $\widetilde{T}={\sf Sl}_{\{8\}}(T)_7$, we have
\[\begin{picture}(430,85)
\put(0,45){$T=$}
\put(25,65){$\tableauL{ \ & \ & \ & \ \\ & \  & \ & \ \ &\\ &  &  \ &\ \\ & &  & \ }$
}
\put(48,42){$2 5$}
\put(58,23){$1 3 6 8$}
\put(76,03){$4 7 9 10$}
\put(145,45){$\widetilde{T}=$}
\put(170,65){$\tableauL{ \ & \ & \ & 4 \\ &  1 & 3 & 7\\ &  &  6 & 8 \\ & &  & 9 }$}
\put(193,42){$2 5$}
\put(231,03){$10$}
\put(285,45){${\sf shift}_{\{8\}}(\widetilde{T})=$}
\put(340,65){$\tableauL{ \ & \ & \ & 4 \\ &  1 & 3 & 7\\ &  &  6 & 9 \\ & &  & 10 }$}
\put(363,42){$2 5$}
\put(384,23){$8$}
\put(420,7){$.$}
\end{picture}\]      
\end{example}

\begin{proposition}\label{prop:sameJDT}
Suppose $T\in \Tab(\rho_n / \rho_n,|\rho_{n,m}|)$ is \good{} and $h\in[n-1]$ such that for $i\geq h$, $\mathcal{I}(T)_i=\emptyset$.
Let $J\subseteq I\subseteq E_{h}(T)$, where $\min I\in J$. 
Then 
\[{\sf Sl}_{I\setminus J}(T)_{k'}={\sf shift}_J ({\sf Sl}_I(T)_{k'}),\] for $k'< k=\min\{i  \ | \ {\sf Sl}_{I\setminus J}(T)_{i}(h,h)=\min I\}$.
\end{proposition}

\begin{proof}
We proceed by induction on $k'$. 
Let $\min I=j$ and
\[M=\min\{i\in E_{h+1}({\sf Sl}_{I\setminus J}(T)) \ | \ i>j\}.\]
For $k'=0$, the result is trivial by the definitions of ${\sf shift}_{J}$ and ${\sf Sl}_I$.
Suppose the statement holds for some $k'-1$ where $k'<k$.  
Let ${\sf path}({{\sf Sl}_{I\setminus J}}(T)_{k'-1})$
be the sequence of boxes that $\bullet$ occupies in computing ${{\sf Sl}_{I\setminus J}}(T)_{k'}$ from ${{\sf Sl}_{I\setminus J}}(T)_{k'-1}$. 
By the inductive assumption and definition of ${\sf shift}_J$, ${{\sf Sl}_{I\setminus J}}(T)_{k'-1}$ and ${\sf Sl}_I(T)_{k'-1}$ may only differ at edge labels in column $h$ and labels weakly below $M$ in column $h+1$. 
Thus we need only consider when ${\sf path}({\sf Sl}_{I\setminus J}(T)_{k'-1})$ intersects columns $h$ and $h+1$.

By the definitions of $k$ and $M$, if $M\in {\sf Sl}_{I\setminus J}(T)_{k'-1}(h+1,h+1)$ or $M\in E_{h+1}({\sf Sl}_{I\setminus J}(T)_{k'-1})$ the result follows by the inductive assumption.
Thus assume $M\in {\sf Sl}_{I\setminus J}(T)_{k'-1}(r,h+1)$ for some $r\in[h]$.
Let 
\[y=\min\{r' \ | \ (r',h+1)\in {\sf path}({\sf Sl}_{I\setminus J}(T)_{k'-1})\}.\]
\noindent
	{\sf Case 1:} ($y$ does not exist or $y\geq r$).
If $y$ does not exist, the result follows by the inductive assumption and definition of ${\sf shift}_J$. Otherwise this implies that ${\sf Sl}_{I\setminus J}(T)_{k'}(y,h)=m'$, where $m'\geq M>j$ since $y\geq r$. Since $k'<k$, $j\in E_{h}({\sf Sl}_{I\setminus J}(T)_{k'})$. Thus ${\sf Sl}_{I\setminus J}(T)_{k'}$ is not increasing down columns, so we cannot have $y\geq r$. 

\noindent
	{\sf Case 2:} ($y<r$).
	By assumption for $i\geq h$, $\mathcal{I}(T)_i=\emptyset$, so $E_{i}(T)=E_{i}(U_{n,m})$ for $i\geq h$.
	Thus this rectification simply moves up those entries in column $h+1$ below row $y$. The result follows by the inductive assumption and definition of ${\sf shift}_J$.
\end{proof}

We say $T\in \Tab(\rho_n / \rho_n,|\rho_{n,m}|)$ is $r$-\emph{compatible} if for each $j=S_{\rho_{n,m}}(r',c)$ with $r'\in[r]$
\begin{equation}\label{eq:nCompat}
    j\not\in T_i(c'-2,c'), \text{ for any } c'>c
\end{equation}
 where $i\in [\binom{n+1}{2}]$.
For $r\in[m]$, let \[{\sf row}_r(S_{\rho_{n,m}}):=\{\text{entries in row } r \text{ of } S_{\rho_{n,m}} \}=\big\{j+ \sum_{i=0}^{r-2}n-i \ | \ j\in[n-r+1]\big\}.\]
For $T\in \Tab(\rho_n / \rho_n,|\rho_{n,m}|)$ \good{} and $k\in[n-1]$, we say $\mathcal{I}(T)_k$ is
\emph{slidable} if for 
\begin{align}
    T^{(k)}&:=({\sf Sl}_{\mathcal{I}(T)_{k-1}}\circ\ldots\circ{\sf Sl}_{\mathcal{I}(T)_1})(U_{n,m}),\nonumber  \\
    \mathcal{I}(T)_k & \subseteq \bigcup_{i=1}^k\{\min (E_{k}(T^{(k)})\cap {\sf row}_i(S_{\rho_{n,m}}))\}\label{eqn:n-slidable}.
\end{align}
If $\mathcal{I}(T)_k$ is slidable for each $k\in[n-1]$, we say $\mathcal{I}(T)$ is slidable.
\begin{example} For $T$ as below, we compute the right hand side of Equation (\ref{eqn:n-slidable}) for $k=3$:
\[\bigcup_{k=1}^3\{\min (E_{3}(T)\cap {\sf row}_k(S_{\rho_{4,4}}))\}=\bigcup_{k=1}^3\{\min (\{1,3,5,6,8\}\cap {\sf row}_k(S_{\rho_{4,4}}))\}=\{1,5,8\}.\]
To the right we have $S_{\rho_{4,4}}$ with the entries of $E_{3}(T)$ in shaded where the slidable entries $\{1,5,8\}$ are shaded lighter. By  Equation (\ref{eqn:n-slidable}), any subset $I$ of the lighter shaded fillings is slidable, so $\{1,8\}$ is slidable but $\{1,3\}$ is not.
        \[\begin{picture}(250,90)
\put(0,45){$T=$}
\put(25,65){$\tableauL{ \ & \ & \ & \ \\ & \  & \ & \ \ &\\ &  &  \ &\ \\ & &  & \ }$}
\put(50,43){$2$}
\put(53,23){$1 3 5 6 8$}
\put(78,3){$4 7 9 10$}
\put(150,65){$\ytableausetup
{boxsize=1.6em}\begin{ytableau}
 *(ltrBlue)1 & 2  & *(ltBlue) 3  & 4\\
\none  & *(ltrBlue)5 & *(ltBlue) 6 & 7\\
\none  & \none & *(ltrBlue)8 & 9\\
 \none & \none & \none & 10 
\end{ytableau}$}
\end{picture}\]   
\end{example}
\begin{proposition}\label{prop:sljSyncAtKSlidable}
Suppose $T\in \Tab(\rho_n / \rho_n,|\rho_{n,m}|)$ is \good{} and $h\in[n-1]$ 
such that for $i\geq h$, $\mathcal{I}(T)_i=\emptyset$. 
Let $J\subseteq I\subseteq E_{h}(T)$ denote the minimal $\# J$ elements of $I$.
Then if $J$ is slidable and $T$ is $r-1$-compatible where $\max J\in {\sf row}_r(S_{\rho_{n,m}})$:
\begin{enumerate}
    \item ${\sf Sl}_I(T)_{k}={\sf Sl}_{I\setminus J}(T)_{k}$, where $k=\min\{i  \ | \ \max J={\sf Sl}_{I\setminus J}(T)_i(h,h)\}$, and 
    \item ${\rect}({\sf Sl}_{I}(T))={\rect}({\sf Sl}_{I\setminus J}(T))$.
\end{enumerate}
\end{proposition}
\begin{proof}
We proceed by induction on $\#J=\ell$. If $J=\emptyset$, the result is trivial. Consider $\#J=\ell\geq 1$ and suppose the statement holds for $\#J<\ell$. Let $j=\max J$. Using the inductive assumption, it suffices to prove the statement when $J=\{j\}$. By Proposition \ref{prop:sameJDT} and the definition of $k$, it follows that ${\sf Sl}_I(T)_{k-1}(n-s,n)=j$ for some $s\geq 1$. 
We will first show $s=1$.

For $i\in[n]$, let $a_{i}=S_{\rho_{n,m}}(i,h)$,  $m_{i}=S_{\rho_{n,m}}(i,h+1)$, and if $h+2\leq n$, $q_{i}=S_{\rho_{n,m}}(i,h+2)$. 
Since $T$ is $r-1$-compatible, 
$a_{i}\in {\sf col}_{h}({\sf Sl}_{I\setminus\{j\}}(T)_{k})$ for $i\in[r-1]$ by the definition of $j$.
Since $\mathcal{I}(T)_i=\emptyset$ for $i\geq h$, $m_{i}\in {\sf col}_{h+1}({\sf Sl}_{I\setminus\{j\}}(T)_{k})$
and $q_{i}\in {\sf col}_{h+2}({\sf Sl}_{I\setminus\{j\}}(T)_{k})$
for $i\in[n]$. Thus by the minimality of $j$ in $I$ and the definition of $k$, ${\sf Sl}_{I\setminus\{j\}}(T)_{k}$ has the form below, where $r_i:=r-i$. By Proposition \ref{prop:sameJDT} and the definition of $k$, showing $s=1$ is equivalent to showing $j>c$.

\begin{equation*}\label{eq:almostSynced}
\begin{picture}(120,150)
\put(0,130){$\ytableausetup
{boxsize=2.1em}\begin{ytableau}
\ddots &  \vdots & \vdots & \vdots & \ddots\\
\ddots &  a_{r_i}  & m_{r_i} & q_{r_i} & \ddots\\
\ddots &  b & c & q_{r_{i}+1}  & \ddots\\
\none  &  j & d & q_{r_{i}+2} & \ddots\\
\none & \none & e & q_{r_{i}+3} & \ddots\\
\none & \none & \none & \ddots & \ddots
\end{ytableau}$
}
\end{picture}
\end{equation*}

If $b$ does not exist, by the definition of row rectification and since $\mathcal{I}(T)_i=\emptyset$ for $i\geq h$, it follows that $j>c$.
Otherwise, since $\{j\}$ is slidable, $b\in {\sf row}_{r'}(S_{\rho_{n,m}})$ where $r'<r$. 
Since $a_{r_1}\in {\sf col}_{h}({\sf Sl}_{I\setminus\{j\}}(T)_{k})$ and $a_{r_i}\leq a_{r_1}<j$, this forces $b=a_{r_1}$.
Since $m_{r_1}\in {\sf col}_{h+1}({\sf Sl}_{I\setminus\{j\}}(T)_{k})$ this implies that $c=m_{r_1}$, so $j>c$. 

Thus ${\sf Sl}_I(T)_{k-1}(n-1,n)=j$. Then (1) follows by Proposition \ref{prop:sameJDT} applied to $k-1$ and the definition of ${\sf jdt}$. Lastly (2) follows directly from (1).
\end{proof}

\begin{claim}\label{claim:compatSl}
Suppose  $T\in \Tab(\rho_n / \rho_n,|\rho_{n,m}|)$ is \good{}.
If $\mathcal{I}(T)_i\cap \bigcup_{k=1}^{r} {\sf row}_k(S_{\rho_{n,m}})$ is slidable for each $i\in[n-1]$, then $T$ is $r$-compatible.
\end{claim}
\begin{proof}
Let $J_i^r=\mathcal{I}(T)_i\cap \bigcup_{k=1}^{r} {\sf row}_k(S_{\rho_{n,m}})$.
We proceed by induction on $r$. For $r=0$ the result is trivial. Suppose the result holds for some $r-1\geq 0$ and $J_i^r$ is slidable for each $i\in[n-1]$. 
By the inductive assumption, $T^{(\ell)}$ is $r-1$-compatible for each $\ell\in[n-1]$. By inducting again on $\ell$ and applying Proposition \ref{prop:sljSyncAtKSlidable}, no entries in $J_i^r$ violate the compatibility condition in $T^{(\ell)}$, so the result follows.
\end{proof}

\begin{proposition}\label{prop:countCompat}
 Suppose $T\in \Tab(\rho_n / \rho_n,|\rho_{n,m}|)$. Then 
${\rect}(T)=S_{\rho_{n,m}}$ if and only if $T$ is \good{} such that $\mathcal{I}(T)$ is slidable.
\end{proposition}
\begin{proof}
$(\Leftarrow)$ 
 Using Claim \ref{claim:compatSl}, Proposition \ref{prop:sljSyncAtKSlidable} for $\mathcal{I}(T)_i=J$ for $i\in[n-1]$ gives
\[{\rect}(T)={\rect}(T^{(n-1)})={\rect}(T^{(n-2)})=\ldots={\rect}(T^{(1)})={\rect}(U_{n,m})=S_{\rho_{n,m}}.\]

$(\Rightarrow)$  
If $T$ is \bad{}, ${\rect}(T)\neq S_{\rho_{n,m}}$  by Proposition \ref{prop:badEdgeOrder}, so assume $T$ is \good{}. 
Let $j=S_{\rho_{n,m}}(r,c)$ be minimal in $[\binom{n+1}{2}]$ such that $j\in \mathcal{I}(T)_{i'}$ for some $i'\in[n-1]$ where $\{j\}$ not slidable. 
{By Claim \ref{claim:compatSl}, Proposition \ref{prop:sljSyncAtKSlidable}, and the minimality of $j$, without loss of generality, we may assume if $j'=S_{\rho_{n,m}}(r',c')$ for $r'<r$, $j'\not\in \mathcal{I}(T)_i$ for $i\in[n-1]$.} Thus
by the definition of row rectification, it suffices to consider the case when $r=1$.
This gives that $\mathcal{I}(T)_{i'}\cap[n]$ is not slidable.

Consider the set of all possible $\mathcal{I}(T')$ for \good{} $T'\in \Tab(\rho_n / \rho_n,|\rho_{n,m}|)$.
For each such $T'$, let $J(T')=(J_1',J_2',\ldots,J_{n-1}')$, where $J_i':=\mathcal{I}(T')_i'\cap [n]$. By the uniqueness of $\mathcal{I}(T')$, distinct $J(T')$ give distinct labelings of the diagonal edges of $\rho_n$ with $[n]$.

For each $i\in[n-1]$, there are 2 ways to choose slidable ${J}_i'$ for $T'$: \[{J}_i'=\{\min E_i(T'^{(i)})\} \text{ \ or \  }{J}_i'=\emptyset.\]
Thus there are $2^{n-1}$ possible slidable choices for $J(T')$.

Let $\Phi$ denote the restriction of $T'$ to entries in $[n]$. By the previous direction, if $\mathcal{I}(T')$ is slidable, ${\rect}(T')= S_{\rho_{n,m}}$, so ${\rect}(\Phi(T'))= S_{(n)}$. 
By Theorem \ref{thm:locPieri}, $d_{\rho_n,(n)}^{\rho_n}=\binom{\ell(\rho_n)}{n}2^{n-1}=2^{n-1}$. 
 Since there are $2^{n-1}$ slidable $J(T')$, these $J(T')$ precisely construct the labelings of the diagonal edges of $\rho_n$ with $[n]$ that count $d_{\rho_n,(n)}^{\rho_n}$.
 Thus by the pigeonhole principle, since the labelings of $[n]$ determined by these $2^{n-1}$ slidable $J(T')$ are distinct and $J(T)$ is not slidable, ${\rect}(\Phi(T))\neq S_{(n)}$, so ${\rect}(T)\neq S_{\rho_{n,m}}$.
\end{proof}

\begin{example}\label{ex:slideIllustration} Below are constructions of $T$ and $T'$ with  $\mathcal{I}(T)_k$ and $\mathcal{I}(T')_k$ given above each arrow. Beneath each arrow, entries of $E_k(T)$ are shaded where those entries in the right hand side of Equation (\ref{eqn:n-slidable}), \emph{i.e.} the slidable entries,
shaded lighter. 
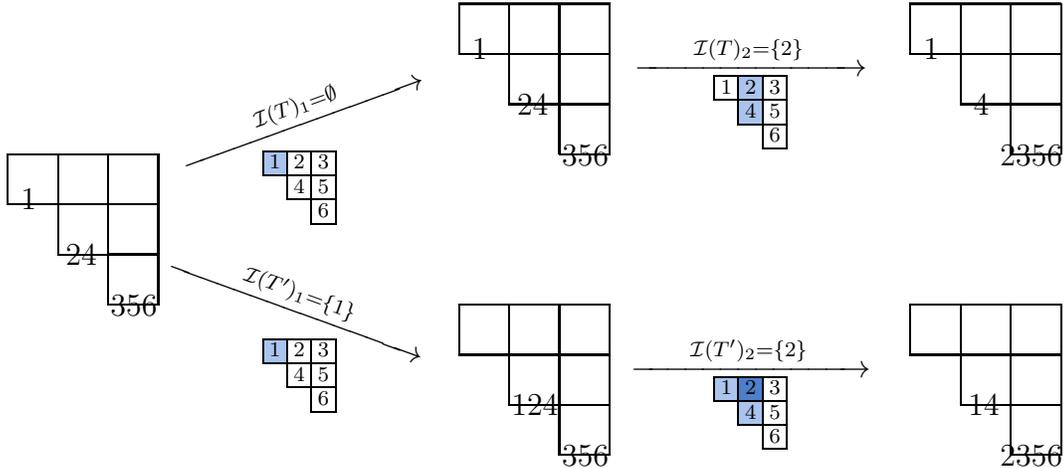
\begin{center}
\begin{tikzpicture}
\node  at (0,0) {\begin{tikzpicture}
\node at (0,0) {$\begin{picture}(50,60)
\put(0,45){$\tableauL{ \ & \ & \ \\ & \  & \  \\ &  &  \ }$}
\put(5,43){$1$}
\put(22,22){$2 4$}
\put(39,3){$3 5 6$}
\end{picture}$};
\end{tikzpicture}};
\node at (3,1.75){\rotatebox{20}{$\xrightarrow{\hspace*{1cm}\mathcal{I}(T)_1=\emptyset\hspace*{1cm}}$}};
\node at (3,0.75){\tiny$\ytableausetup
{boxsize=1.1em}\begin{ytableau}
*(ltrBlue)1  &  2  & 3  \\
\none & 4 & 5 \\
 \none & \none & 6
\end{ytableau}$};
\node  at (6,2) {\begin{tikzpicture}
\node at (0,0) {$\begin{picture}(50,60)
\put(0,45){$\tableauL{ \ & \ & \ \\ & \  & \  \\ &  &  \ }$}
\put(5,43){$1$}
\put(22,22){$2 4$}
\put(39,3){$3 5 6$}
\end{picture}$};
\end{tikzpicture}};
\node at (3,-0.75){\rotatebox{340}{$\xrightarrow{\hspace*{0.9cm}\mathcal{I}(T')_1=\{1\}\hspace*{0.9cm}}$}};
\node at (3,-1.75){\tiny$\ytableausetup
{boxsize=1.1em}\begin{ytableau}
*(ltrBlue) 1 &  2  & 3  \\
\none & 4 & 5 \\
 \none & \none & 6
\end{ytableau}$};
\node  at (6,-2) {\begin{tikzpicture}
\node at (0,0) {$\begin{picture}(50,60)
\put(0,45){$\tableauL{ \ & \ & \ \\ & \  & \  \\ &  &  \ }$}
\put(20,22){$1 2 4$}
\put(39,3){$3 5 6$}
\end{picture}$};
\end{tikzpicture}};

\node at (9,2.5){$\xrightarrow{\hspace*{0.7cm}\mathcal{I}(T)_2=\{2\}\hspace*{0.7cm}}$};
\node at (9,1.75){\tiny$\ytableausetup
{boxsize=1.1em}\begin{ytableau}
1 &  *(ltrBlue)2  & 3  \\
\none & *(ltrBlue)4 & 5 \\
 \none & \none & 6
\end{ytableau}$};
\node  at (12,2) {\begin{tikzpicture}
\node at (0,0) {$\begin{picture}(50,60)
\put(0,45){$\tableauL{ \ & \ & \ \\ & \  & \  \\ &  &  \ }$}
\put(5,43){$1$}
\put(24,22){$4$}
\put(34,3){$2 3 5 6$}
\end{picture}$};
\end{tikzpicture}};

\node at (9,-1.5){$\xrightarrow{\hspace*{0.7cm}\mathcal{I}(T')_2=\{2\}\hspace*{0.7cm}}$};
\node at (9,-2.25){\tiny$\ytableausetup
{boxsize=1.1em}\begin{ytableau}
*(ltrBlue)1  &  *(ltBlue)2  & 3  \\
\none & *(ltrBlue)4 & 5 \\
 \none & \none & 6
\end{ytableau}$};
\node  at (12,-2) {\begin{tikzpicture}
\node at (0,0) {$\begin{picture}(50,60)
\put(0,45){$\tableauL{ \ & \ & \ \\ & \  & \  \\ &  &  \ }$}
\put(22,22){$1 4$}
\put(34,3){$2 3 5 6$}
\end{picture}$};
\end{tikzpicture}};
\end{tikzpicture}
\end{center}
Thus $\mathcal{I}(T)_k$, or $\mathcal{I}(T')_k$, is slidable if and only if all entries of $\mathcal{I}(T)_k$, or $\mathcal{I}(T')_k$, are shaded lightly. Here
$\mathcal{I}(T)_1, \mathcal{I}(T')_1$, and $\mathcal{I}(T)_2$ are all slidable, but $\mathcal{I}(T')_2$ is not. Therefore by Proposition \ref{prop:countCompat}, ${\rect}(T)=S_{\rho_{3,3}}$, but ${\rect}(T')\neq S_{\rho_{3,3}}$.
\end{example}

\begin{theorem}\label{thm:Local2n2}
 $d_{\rho_n,\rho_{n,m}}^{\rho_{n}}=2^{\binom{n}{2}-\binom{n-m}{2}}$.
\end{theorem}
\begin{proof}
By Proposition \ref{prop:countCompat} $d_{\rho_n,\rho_{n,m}}^{\rho_n}$ counts the slidable 
$\mathcal{I}(T)$ over $T\in \Tab(\rho_n / \rho_n,|\rho_{n,m}|)$.
 Let
 $J^k=(J^k_1,J^k_2,\ldots,J^k_{n-1})$ be such that $J^k_i=\mathcal{I}(T)_i\cap [{\sf row}_k(S_{\rho_{n,m}})]$.
 Using the same argument presented in the proof of Proposition \ref{prop:countCompat}, there are $d_{\rho_{n-k},(n-k)}^{\rho_{n-k}}=2^{n-k}$ choices for slidable $J^k$.
Summing over $k\in[m]$, this gives $2^{\binom{n}{2}-\binom{n-m}{2}}$ ways to choose slidable $\{J^k\}_{k\in[m]}$ and thus $2^{\binom{n}{2}-\binom{n-m}{2}}$ ways to choose slidable $\mathcal{I}(T)$.
\end{proof}

\begin{proof}[Proof of Theorem \ref{thm:locCoeffSkewSC}]
Suppose $\ell(\mu)=m\leq\ell(\lambda)=n$ such that for $\rho_{n,m}\subseteq \mu$. 
Observe that $|\mathcal{E}_{\rho_n}(\mu)| = 1$ if $\mu\subseteq\rho_n$ and $0$ otherwise.
Thus by Lemma \ref{lem:locAF}, when $\mu\subseteq\rho_n$, \emph{i.e.} when $\mu=\rho_{n,m}$,
\begin{equation*}
{\mathfrak d}_{\lambda,\rho_{n,m}}^{\lambda}= 2^{|\rho_{n,m}|-\ell(\rho_{n,m})}=
2^{|\rho_{n,m}|-m}=
2^{\binom{n}{2}-\binom{n-m}{2}}
\end{equation*}
 and otherwise ${\mathfrak d}_{\lambda,\mu}^{\lambda}=0$. 
This proves the rightmost equality.

By the definition of $d_{\lambda,\mu}^{\lambda}$, if $\mu\not\subseteq\rho_n$, $d_{\lambda,\mu}^{\lambda}=0$.
Similarly, it is straightforward to see that $d_{\lambda,\rho_{n,m}}^{\lambda}=d_{\rho_n,\rho_{n,m}}^{\rho_n}$ since $m\leq \ell(\lambda)$.
Thus the result follows by Theorem \ref{thm:Local2n2}.
\end{proof}

\begin{example}
Below we illustrate the bijection between the shading of $\lambda$ and the tableau $T$ from Example \ref{ex:shadeBij}. The shaded subsets of the first $n-1$ columns of $\rho_{n,m}$ define a choice of slidable $\mathcal{I}(T)$, where box $(i,r)$ is shaded if and only if $\min E_i(T^{(i)})\cap [{\sf row}_r(S_{\rho_{n,m}})]\in \mathcal{I}(T)_i$. 
Below each arrow, we illustrate how the choice of $\mathcal{I}(T)_i$ affects the shading.
	
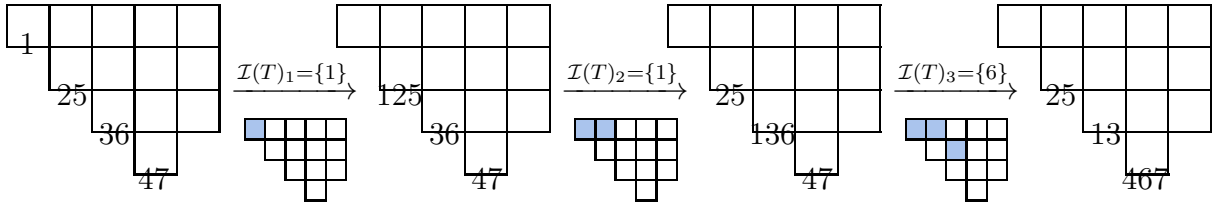
\[\begin{picture}(450,80)
\put(0,55){$\ytableausetup
{boxsize=1.3em}\begin{ytableau}
 \ & \  & \  & \ & \  \\
\none  & \ & \ & \ & \ \\
\none  & \none & \ & \ & \ \\
 \none & \none & \none & \ 
\end{ytableau}$}
\put(5,52){$1$}
\put(19,33){$2 5$}
\put(35,18){$3 6 $}
\put(50,1){$4 7 $}
\put(85,35){$\xrightarrow{\mathcal{I}(T)_1=\{1\}}$}
\put(90,20){$\ytableausetup
{boxsize=0.6em}\begin{ytableau}
 *(ltrBlue)\ & \  & \  & \ & \  \\
\none  & \ & \ & \ & \ \\
\none  & \none & \ & \ & \ \\
 \none & \none & \none & \ 
\end{ytableau}$}
\put(125,55){$\ytableausetup
{boxsize=1.3em}\begin{ytableau}
 \ & \  & \  & \ & \  \\
\none  & \ & \ & \ & \ \\
\none  & \none & \ & \ & \ \\
 \none & \none & \none & \ 
\end{ytableau}$}
\put(140,33){$1 2 5$}
\put(160,18){$3 6 $}
\put(176,1){$4 7 $}
\put(210,35){$\xrightarrow{\mathcal{I}(T)_2=\{1\}}$}
\put(215,20){$\ytableausetup
{boxsize=0.6em}\begin{ytableau}
 *(ltrBlue)\ & *(ltrBlue)\  & \  & \ & \  \\
\none  & \ & \ & \ & \ \\
\none  & \none & \ & \ & \ \\
 \none & \none & \none & \ 
\end{ytableau}$}
\put(250,55){$\ytableausetup
{boxsize=1.3em}\begin{ytableau}
 \ & \  & \  & \ & \  \\
\none  & \ & \ & \ & \ \\
\none  & \none & \ & \ & \ \\
 \none & \none & \none & \ 
\end{ytableau}$}
\put(268,33){$2 5$}
\put(281,18){$1 3 6 $}
\put(301,1){$4 7 $}
\put(335,35){$\xrightarrow{\mathcal{I}(T)_3=\{6\}}$}
\put(340,20){$\ytableausetup
{boxsize=0.6em}\begin{ytableau}
 *(ltrBlue)\ & *(ltrBlue)\  & \  & \ & \  \\
\none  & \ & *(ltrBlue)\ & \ & \ \\
\none  & \none & \ & \ & \ \\
 \none & \none & \none & \ 
\end{ytableau}$}
\put(375,55){$\ytableausetup
{boxsize=1.3em}\begin{ytableau}
 \ & \  & \  & \ & \  \\
\none  & \ & \ & \ & \ \\
\none  & \none & \ & \ & \ \\
 \none & \none & \none & \ 
\end{ytableau}$}
\put(393,33){$2 5$}
\put(410,18){$1 3 $}
\put(422,1){$4 6 7 $}
\end{picture}\] 
\end{example}

\section*{Acknowledgements}
We thank Harshit Yadav and Alexander Yong for helpful  conversations. This material is based upon work of the author supported by the National Science Foundation Graduate Research Fellowship Program under Grant No. DGE -- 1746047 and the National Science Foundation Research Training Grant No. DMS 1937241.

\end{document}